\documentclass[reqno]{amsart}

\usepackage[margin=1.4in]{geometry}
\usepackage{amssymb, amsmath, amsthm, amsfonts, amscd}
\usepackage[shortlabels]{enumitem}
\usepackage{tikz, tikz-cd}
\usepackage[all]{xy}
\usepackage[new]{old-arrows}
\usepackage{pst-node}
\usepackage{varwidth}
\usepackage{graphicx}
\usepackage{romannum}
\usepackage{subfig}
\usepackage{tabularx, booktabs}
\usepackage{multirow, multicol}
\usepackage{xcolor}
\usepackage{mathrsfs}
\usepackage{yfonts, euscript}
\usepackage{csquotes, dirtytalk}
\usepackage[pagewise]{lineno}
\usepackage[colorlinks=true, allcolors=blue]{hyperref}

\usetikzlibrary{arrows, arrows.meta}

\newtheorem{thm}{{\bf Theorem}}[section]
\newtheorem{lemma}[thm]{{\bf Lemma}}
\newtheorem{prop}[thm]{{\bf Proposition}}

\newtheorem*{rmk}{{\bf Remark}}

\newcommand{\RNum}[1]{\uppercase\expandafter{\romannumeral #1\relax}}

\begin{document}


\title[Normalizer of Twisted Chevalley Groups]{Normalizer of Twisted Chevalley Groups over Commutative Rings}


\author{Shripad M. Garge}
\address{Department of Mathematics, 
	Indian Institute of Technology Bombay,\newline \indent
	Powai, Mumbai 400076, India}
\email{\href{mailto:shripad@math.iitb.ac.in}{shripad@math.iitb.ac.in}, \href{mailto:smgarge@gmail.com}{smgarge@gmail.com}}

\author{Deep H. Makadiya}
\address{Department of Mathematics, 
	Indian Institute of Technology Bombay,\newline \indent
	Powai, Mumbai 400076, India}
\email{\href{mailto:deepmakadia25.dm@gmail.com}{deepmakadia25.dm@gmail.com}}




\begin{abstract}   
    Let $R$ be a commutative ring with unity. Consider the twisted Chevalley group $G_{\pi, \sigma} (\Phi, R)$ of type $\phi$ over $R$ and its elementary subgroup $E'_{\pi, \sigma} (\Phi, R)$. 
    This paper investigates the normalizers of $E'_{\pi, \sigma}(\Phi, R)$ and $G_{\pi, \sigma}(\Phi, R)$ in the larger group $G_{\pi, \sigma}(\Phi, S)$, where $S$ is an extension ring of $R$. 
    We establish that under certain conditions on $R$ these normalizers coincide.  
    Moreover, in the case of adjoint type groups, we show that they are precisely equal to $G_{\pi, \sigma}(\Phi, R)$.
\end{abstract}


\maketitle 

\pagenumbering{arabic}


\section{Introduction}

Let $\mathcal{L}$ be a (finite-dimensional) semisimple Lie algebra over $\mathbb{C}$ with root system $\Phi$, and let $\pi$ be a (finite-dimensional) representation of $\mathcal{L}$.  
Consider a commutative ring $R$ with unity, and let $G_{\pi} (\Phi, R)$ be the Chevalley group of type $\Phi$ over $R$.  
Let $\sigma$ be an automorphism of $G_{\pi} (\Phi, R)$, of order $2$, which is the product of a graph automorphism and a ring automorphism. 
The subgroup of \(G_{\pi} (\Phi, R)\) consisting of elements fixed under \(\sigma\) is known as the twisted Chevalley group of type \(\Phi\) over \(R\), denoted by \(G_{\pi, \sigma} (\Phi, R)\). Let $E'_{\pi, \sigma} (\Phi, R)$ be its elementary subgroup, known as a elementary twisted Chevalley of type $\Phi$ over a ring $R$.

The primary goal of this paper is to establish results that will be utilized in the author’s forthcoming paper on the automorphisms of twisted Chevalley groups over commutative rings. 
Specifically, we prove that under certain conditions on $R$, the normalizer of $G_{\pi, \sigma} (\Phi, R)$ and $E'_{\pi, \sigma} (\Phi, R)$ in the larger group $G_{\pi, \sigma} (\Phi, S)$ are coincides, where $S$ is a ring extension of $R$.  
Moreover, in the case of adjoint type groups, we show that these normalizers are precisely equal to \( G_{\pi, \sigma} (\Phi, R) \).  
The main theorem of this paper is stated below.  

\begin{thm}[Main Theorem]\label{thm:normalizer_of_G_and_E}
    Let \(R\) be a Noetherian commutative ring with unity.  
    Let \(G_{\pi, \sigma} (\Phi, R)\) be a twisted Chevalley group of type \({}^2 A_n \ (n \geq 3), {}^2 D_n \ (n\geq 4)\), or \({}^2 E_6\), and let \(E'_{\pi, \sigma} (\Phi, R)\) be its elementary subgroup.  
    Assume that \(1/2 \in R\) and, if \(\Phi_\rho \sim {}^2 A_{2n}\), additionally assume that \(1/3 \in R\).  
    Furthermore, suppose there exists an invertible element \(a \in R\) such that \(\theta(a) = - a\).  
    If \(S\) is a ring extension of \(R\), then we have  
    \[
        N_{G_{\pi,\sigma}(\Phi, S)}(G_{\pi, \sigma}(\Phi, R)) = N_{G_{\pi,\sigma} (\Phi, S)} (E'_{\pi, \sigma} (\Phi, R)).
    \]
    Moreover, if \(G\) is of adjoint type, we have  
    \[
        N_{G_{\text{ad},\sigma}(\Phi, S)}(G_{\text{ad}, \sigma}(\Phi, R)) = N_{G_{\text{ad},\sigma} (\Phi, S)} (E'_{\text{ad}, \sigma} (\Phi, R)) = G_{\text{ad}, \sigma} (\Phi, R).
    \]
\end{thm}

The remainder of this paper is structured as follows.  
Section~\ref{sec:preliminaries} provides a review of essential background material and relevant known results on Chevalley groups and twisted Chevalley groups.  
In Section~\ref{sec:TA of ETCG}, we describe the tangent algebra of the elementary twisted Chevalley group of adjoint type.  
Section~\ref{sec:key lemma} presents a key lemma that serves as a crucial step in proving the main result.  
Finally, Section~\ref{sec:normalizer_of_G_and_E} contains the proof of the main theorem.  


\section{Preliminaries}\label{sec:preliminaries}


We assume the reader has prior knowledge of semisimple Lie algebras, root systems, and the Weyl group, including their fundamental properties. In this section, we formally define Chevalley and twisted Chevalley groups. For a more detailed discussion, we refer the reader to \cite{EA1}, \cite{EB24:final}, \cite{RC}, \cite{SG&DM1}, \cite{EP&NV}, \cite{RS}, \cite{RSTCG} or \cite{NV1}.


\subsection{Chevalley Algebras and Groups}\label{subsec:CG}

Let $\mathcal{L} = \mathcal{L}(\Phi, \mathbb{C})$ be a finite dimensional semisimple Lie algebra over $\mathbb{C}$ with root system $\Phi$. 
We fix a Cartan subalgebra $\mathcal{H}$ of $\mathcal{L}$ and consider the root space decomposition of $\mathcal{L}$:
\[
    \mathcal{L} = \mathcal{H} \oplus \Big( \coprod_{\alpha \in \Phi} \mathcal{L}_{\alpha} \Big),
\]
where $\mathcal{L}_\alpha = \{ X \in \mathcal{L} \mid [H,X] = \alpha (H) X, \ \forall H \in \mathcal{H} \}$ for each $\alpha \in \Phi$. 


\subsubsection{}\label{subsubsec:Chevalley algebra}

Fix a simple system \(\Delta\) of \(\Phi\). Consider a Chevalley basis of \(\mathcal{L}\) (for a precise definition, see page 147 of \cite{JH} or page 7 of \cite{RS}) given by  
\[
\{ H_i = H_{\alpha_i}, X_{\alpha} \mid \alpha_i \in \Delta, \, \alpha \in \Phi \}.
\]  
Define \(\mathcal{L}(\Phi, \mathbb{Z})\) as the \(\mathbb{Z}\)-span of this Chevalley basis.  

Now, let \(R\) be a commutative ring with unity. The \textbf{Chevalley algebra} of type $\Phi$ over a ring \(R\) is defined as  
\[
    \mathcal{L}(\Phi, R) = R \otimes_{\mathbb{Z}} \mathcal{L}(\Phi, \mathbb{Z}).
\]


\subsubsection{} 

Consider a (finite-dimensional) faithful representation $\pi: \mathcal{L} \longrightarrow \text{GL}(V)$. 
This induces a natural action of $\mathcal{U}$ and, consequently, of $\mathcal{U}_\mathbb{Z}$ on $V$. 
Note that $V$ contains a lattice $M$ that is invariant under the action of $\mathcal{U}_\mathbb{Z}$, called an \textbf{admissible lattice}. 
Let $\mathcal{L}_{\pi} (\mathbb{Z})$ denote the stabilizer of $M$ in $\mathcal{L}$.
It follows that $\mathcal{L}_{\pi} (\mathbb{Z})$ is a lattice in $\mathcal{L}$ and can be expressed as 
$$
\mathcal{L}_{\pi} (\mathbb{Z}) = \mathcal{H}_{\pi} (\mathbb{Z}) \oplus \coprod_{\alpha \in \Phi} \mathbb{Z} X_{\alpha},
$$
where 
$$
\mathcal{H}_{\pi} (\mathbb{Z}) = \mathcal{H} \cap \mathcal{L}_{\pi} (\mathbb{Z}) = \{ H \in \mathcal{H} \mid \mu(H) \in \mathbb{Z} \text{ for all weights } \mu \text{ of the representation } \pi \}.
$$
This shows that $\mathcal{L}_{\pi} (\mathbb{Z})$ depends only on the weight lattice $\Lambda_\pi$, making the notation independent of the particular choice of $M$.


\subsubsection{}

Let $R$ be a commutative ring with unity. 
Let $V(R) = M \otimes_\mathbb{Z} R$, $\mathcal{L} (R) = \mathcal{L} (\mathbb{Z}) \otimes_{\mathbb{Z}} R$ and $\mathcal{L}_{\pi} (R) = \mathcal{L}_{\pi} (\mathbb{Z}) \otimes_{\mathbb{Z}} R$.
Consider the automorphisms of $V(R)$ of the form $x_\alpha (t): = \exp {(t \pi (X_\alpha))} \ (t \in R, \alpha \in \Phi)$, where 
$$\exp {(t \pi(X_\alpha))} = \sum_{n=0}^{\infty} \frac{t^n \pi(X_\alpha)^n}{n!} .$$ 
The action of $x_\alpha (t)$ on $V(R)$ is the same as the action described in \cite[Chapter 3]{RS}. 
The subgroup of $\operatorname{Aut}(V(R))$ generated by all $x_\alpha(t) \ (t \in R, \alpha \in \Phi)$ is called an \textbf{elementary Chevalley group} and is denoted by $E_\pi (\Phi, R)$. 
For a representation $\pi$, let $\Lambda_\pi$ denote the weight lattice of $\pi$, i.e., the lattice generated by all weights of $\pi$. 
If $\pi$ and $\pi'$ are representations of $\mathcal{L}$ such that $\Lambda_\pi = \Lambda_{\pi'}$, then $E_\pi (\Phi, R) \cong E_{\pi'} (\Phi, R)$.
Let $\Lambda_r$ be the lattice generated by roots and $\Lambda_{sc}$ be the lattice generated by fundamental weights. 
If $\pi$ is such that $\Lambda_\pi = \Lambda_{r}$ (resp., $\Lambda_\pi = \Lambda_{sc}$), then $E_\pi(\Phi, R) = E_{\text{ad}}(\Phi, R)$ (resp., $E_\pi(\Phi, R) = E_{sc}(\Phi, R)$) is called an \textbf{adjoint elementary Chevalley group} (resp., \textbf{universal} (or \textbf{simply connected}) \textbf{elementary Chevalley group}). 

Let $U$ (resp., $U^{-}$) to be the subgroup of $E_\pi(\Phi, R)$ generated by all $x_\alpha(t), \alpha \in \Phi^+ \ (\text{resp., } \alpha \in \Phi^-), t \in R$. Let $H$ be the subgroup generated by all $h_\alpha (t) = w_\alpha (t) w_\alpha(1)^{-1},$ where $w_\alpha (t) = x_\alpha(t) x_{-\alpha}(-t^{-1}) x_\alpha(t), \ t \in R^*$ (the group of units in $R$). If $B$ is the subgroup generated by $U$ and $H$, then $U \cap H = 1$, $U$ is normal in $B$ and $B = UH$. Let $N$ be the subgroup generated by all $w_\alpha (t)$ and $W$ be the Weyl group $W(\Phi)$. Then $H$ is normal in $N$ and $W \cong N/H$ with the map $s_\alpha \mapsto H w_\alpha(1), \forall \alpha \in \Phi$. We sometimes use more precise notation such as $U_\pi (\Phi, R), U(\Phi, R)$ or $U(R)$ instead of just $U$. Similarly, this applies to $U^{-}, H, B$ and $N$.


\subsubsection{}

Let $k$ be an algebraically closed field. Then the semisimple linear algebraic groups over the field $k$ are precisely the elementary Chevalley groups $E_\pi (\Phi, k)$ (see \cite[Chapter 5]{RS}). All these groups can be viewed as subgroups of $GL_n(k)$ defined as a common set of zeros of polynomials of matrix entries $x_{ij}$ with integer coefficients. Note that the multiplication map and the inverse map are also defined by polynomials with integer coefficients. Therefore, these polynomials can be considered as polynomials over an arbitrary commutative ring with unity. 


\subsubsection{}

Let $E_\pi (\Phi, \mathbb{C})$ be an elementary Chevalley group viewed as a subgroup of $GL_n(\mathbb{C})$ defined by zero locus of polynomials $p_1(x_{ij}), \dots ,p_m(x_{ij})$. Note that these polynomials can be chosen to have integer coefficients. Let $R$ be a commutative ring with unity and let us consider the groups $$G(R) = \{ (a_{ij}) \in GL_n (R) \mid \widetilde{p}_1(a_{ij}) = 0, \dots, \widetilde{p}_m(a_{ij}) = 0\},$$
where $\widetilde{p}_1(x_{ij}), \dots , \widetilde{p}_m(x_{ij})$ are polynomials having the same coefficients as $p_1(x_{ij}), \dots ,p_m(x_{ij})$, but considered over a ring $R$. This group is called the \textbf{Chevalley group} $G_\pi (\Phi, R)$ of the type $\Phi$ over the ring $R$. If $\pi$ is a representation such that $\Lambda_{\pi} = \Lambda_r$ then $G_{\pi}(\Phi, R) = G_{ad}(\Phi, R)$ is called an \textbf{adjoint Chevalley group}.  If $\pi$ is a representation such that $\Lambda_\pi = \Lambda_{sc}$ then $G_\pi(\Phi, R) = G_{sc}(\Phi, R)$ is called a \textbf{universal} (or \textbf{simply connected}) \textbf{Chevalley group}.

Note that $E_\pi (\Phi, R) \subseteq G_\pi (\Phi, R)$. If $k$ is an algebraically closed field then $E_\pi (\Phi, k) = G_\pi (\Phi, k)$. But in general, equality may not hold (even for a field). 


\subsubsection{}

The subgroup of diagonal matrices (in the standard basis of weight vectors) of the Chevalley group $G_\pi (\Phi, R)$ is called the \textbf{standard maximal torus} of $G_\pi (\Phi, R)$ and it is denoted by $T_\pi (\Phi, R)$. 
This group is isomorphic to Hom$(\Lambda_\pi, R^*)$ where $R^*$ is the group of units in $R$ and the isomorphism is given as follows: Let $\chi \in $ Hom$(\Lambda_{\pi}, R^*)$ be a character of $\Lambda_\pi$. Let $V_\mu$ be a weight space corresponding to weight $\mu$ of $\pi$ and let $V_{\mu}(R) = (V_\mu \cap M) \otimes_\mathbb{Z} R$. 
Define an automorphism $h(\chi)$ of $V(R)$ given by $$ h(\chi) \cdot v = \chi (\mu) v, $$ where $\mu$ is a weight of $\pi$ and $v \in V_{\mu} (R)$. 
Note that, $h(\chi)$ can be extented to $V(R)$ as $V(R) = \coprod_{\mu \in \Omega_\pi} V_{\mu}(R),$ where $\Omega_\pi$ is the collection of weights corresponding to representation $\pi$. 
Therefore, $$T_\pi (\Phi, R) = \{ h(\chi) \mid \chi \in  \text{Hom}(\Lambda_\pi, R^*) \}.$$ 

Note that $H_\pi (\Phi, R)$ is contained in $T_{\pi}(\Phi, R)$. The element $h(\chi) \in H_\pi (\Phi, R) \subset E_\pi(\Phi, R)$ if and only if $\chi \in \text{Hom}(\Lambda_\pi, R^*)$ can be extented to a character $\chi'$ of $\Lambda_{sc}$, that is, $\chi' \in \text{Hom}(\Lambda_{sc}, R^*)$ such that $\chi'|_{\Lambda_{\pi}} = \chi$. Moreover, $h_\alpha(t) = h(\chi_{\alpha, t}) \ (t \in R^*, \alpha \in \Phi),$ where $$ \chi_{\alpha, t}: \lambda \mapsto t^{\langle \lambda, \alpha \rangle} \ (\lambda \in \Lambda_\pi).$$ Therefore $H_\pi (\Phi, R) = E_\pi (\Phi, R) \cap T_\pi(\Phi, R).$ Consider a subgroup $G_{\pi}^{0}(\Phi, R) = E_\pi (\Phi, R) T_\pi (\Phi, R)$ of $G_\pi(\Phi, R)$. If $R$ is a semilocal ring, then $G_\pi (\Phi, R) = G_\pi^0 (\Phi, R)$ (see \cite[Corollary 2.4]{EA2}). The element $h(\chi)$ acts on $\mathfrak{X}_{\alpha} = \{ x_{\alpha} (t) \mid t \in R \}$ by conjugation as follows: $$ h(\chi) x_\alpha (\zeta) h(\chi)^{-1} = x_\alpha (\chi (\alpha) \zeta).$$ 

\begin{rmk}
    For abusive use of notations, we sometimes write $E(R)$ or $E(\Phi, R)$ instead of $E_\pi(\Phi, R)$, similar for $G_\pi (\Phi, R), G^0_\pi (\Phi, R)$ and $T_\pi (\Phi, R)$.
\end{rmk}


\subsubsection{}

A subgroup $H$ of a group $G$ is called \textit{characteristic}, if it is mapped into itself under any automorphism of $G$. In particular, any characteristic subgroup is normal. If the rank of $\Phi$ is $\geq 1$, then $E_\pi (\Phi, R)$ is a characteristic subgroup of $G_\pi (\Phi, R)$ (see \cite[Theorem 5]{LV}).

A group $G$ is said to be \textit{perfect} if $[G,G] = G$, where $[G,G]$ denotes the commutator subgroup of $G$. If the rank of $\Phi$ is $\geq 1$, then the elementary Chevalley group $E_\pi (\Phi, R)$ is perfect, i.e., $[E_\pi (\Phi, R), E_\pi (\Phi, R)]$ (see \cite[Theorem 5]{LV}).


\subsection{Twisted Root System}\label{Subsec:TRS}

Let $V$ be a finite-dimensional real Euclidean vector space and let $\Phi$ be a crystallographic root system. Let $\Delta$ and $\Phi^+$ be the simple and positive root systems, respectively, with respect to some fixed ordering on $V$. Let $\rho$ be a non-trivial angle preserving permutation of $\Delta$ (such a $\rho$ exists only when $\Phi$ is of type $A_l \ (l \geq 1), D_l \ (l \geq 4), E_6, B_2, F_4$ or $G_2$). Note that the possible order of $\rho$ is either $2$ or $3$, with the latter possible only when $\Phi$ is of type $D_4$. We define an isometry $\hat{\rho} \in GL (V)$ as follows:
\begin{enumerate}
    \item If $\Phi$ has one root length, then define $\hat{\rho}(\alpha)= \rho(\alpha)$ for each $\alpha \in \Delta$.
    \item If $\Phi$ has two root lengths. Then define $\hat{\rho}(\alpha)= \rho(\alpha)/ \sqrt{p}$ for each short root $\alpha \in \Delta$ and $\hat{\rho}(\alpha)= \sqrt{p} \hspace{1mm} \rho(\alpha)$ for each long root $\alpha \in \Delta$, where $p = ||\alpha||^2 / ||\beta||^2$, $\alpha$ is a long root and $\beta$ is a short root. 
\end{enumerate} 

Clearly, the order of $\hat{\rho}$ is the same as that of $\rho$ and $\hat{\rho}$ preserves the sign. Note that $\hat{\rho} w_{\alpha} \hat{\rho}^{-1} = w_{\rho(\alpha)}$, hence $\hat{\rho}$ normalizes $W$. Define $V_\rho = \{ v \in V \mid \hat{\rho}(v)=v \}$ and $W_\rho = \{ w \in W \mid \hat{\rho}w\hat{\rho}^{-1} =w \}$. Let $\hat{\alpha} = 1/o(\rho) \sum_{i=0}^{o(\rho)-1} \hat{\rho}^i (\alpha),$ the average of the elements in the $\hat{\rho}$-orbit of $\alpha$. Then $(\beta, \hat{\alpha})=(\beta, \alpha)$ for all $\beta \in V_\rho$. Hence the projection of $\alpha$ on $V_\rho$ is $\hat{\alpha}$. 

Note that $W_\rho$ acts faithfully on $V_\rho$. Let $J = J_\alpha \subset \Phi$ be the $\rho$-orbit of $\alpha$ and let $W_{J}$ be the group generated by all $w_\beta \ (\beta \in J_\alpha)$. Let $w_{J}$ be the unique element of $W_J$ such that $w_J (P_\alpha) = - P_\alpha$, where $P_\alpha$ is a positive system generated by $J_\alpha$ (such a $w_J$ exists and is of highest length element in $W_J$). Then $w_J|_{V_\rho} = w_{\hat{\alpha}}|_{V_\rho}$ and $w_J|_{V_\rho} \in W_{\rho}$. In fact, $\{ w_{\hat{\alpha}}|_{V_\rho} \mid \alpha \in \Delta \}$ forms a generating set of $W_\rho$. Therefore the group $W_\rho|_{V_\rho}$ is a reflection group. Define $\Tilde{\Phi}_\rho = \{ \hat{\alpha} \mid \alpha \in \Phi \}$ and $\Tilde{\Delta}_\rho = \{ \hat{\alpha} \mid \alpha \in \Delta \}$. Then $\Tilde{\Phi}_\rho$ is the (possibly non-reduced) root system corresponding to the Weyl group $W_\rho|_{V_\rho}$ and $\Tilde{\Delta}_\rho$ is the corresponding simple system. In order to make $\Tilde{\Phi}_\rho$ reduced, we can stick to the set of shortest projections of various directions, and denote it by $\Phi_\rho$. Define an equivalence relation $R$ on $\Phi$ by $\alpha \equiv \beta$ iff $\hat{\alpha}$ is a positive multiple of $\hat{\beta}$. If $\Phi/R$ denotes the collection of all equivalence classes of this relation, then $\Phi_\rho$ is in one-to-one correspondence with $\Phi / R$ by identifying a root $\hat{\alpha}$ of $\Phi_\rho$ with a class $[\alpha]$ of $\Phi / R$. Similarly, there exists a one-to-one correspondence between $\Tilde{\Phi}_\rho$ and $\{ J_\alpha \mid \alpha \in \Phi \}$ by sending a root $\hat{\alpha}$ of $\Tilde{\Phi}_\rho$ to $J_\alpha$. Clearly $-[\alpha] = [-\alpha]$ and $-J_{\alpha} = J_{-\alpha}.$

\begin{lemma}[{\cite[page 103]{RS}}]
    \normalfont
    If $\Phi$ is irreducible, then an element of $\Phi/R$ is the positive system of roots of a system of one of the following types:
    \begin{enumerate}[(a)]
        \item $A_1^n, \hspace{1mm} n=1, 2$ or $3$.
        \item $A_2$ (this occurs only if $\Phi$ is of type $A_{2n}$).
        \item $C_2$ (this occurs if $\Phi$ is of type $C_{2}$ or $F_4$).
        \item $G_2$ (this occurs only if $\Phi$ is of type $G_2$).
    \end{enumerate}
\end{lemma}

If a class $[\alpha]$ in $\Phi/R$ is the positive system of roots of a system of type $X$ (where $X$ is any of the above root systems) then we write $[\alpha] \sim X$. Similarly, if $\Phi \sim X$ then we write $\Phi_\rho \sim {}^n X$ where $n$ is the order of $\rho$. In the following table we describe some root systems $\Phi_\rho$ and $\Tilde{\Phi}_\rho$ after the twist:

\begin{center}
    \begin{tabular}{|c|c|c|c|c|}
        \hline
        \multirow{2}{*}{\textbf{Type}} & \multirow{2}{*}{\textbf{$\Tilde{\Phi}_\rho$}} & \multirow{2}{*}{\textbf{$\Phi_\rho$}} & \multicolumn{2}{c|}{\textbf{Type of Roots}} \\
        \cline{4-5}
        & & & \textbf{Long} & \textbf{Short}  \\
        \hline 
        ${}^2 A_{2n-1} \ (n \geq 2)$ & $C_n$ & $C_n$ & $A_1$ & $A_1^2$ \\
        \hline 
        ${}^2 A_{2n} \ (n \geq 2)$ & $BC_n$ & $B_n$ & $A_1^2$ & $A_2$ \\
        \hline 
        ${}^2 D_{n} \ (n \geq 4)$ & $B_{n-1}$ & $B_{n-1}$ & $A_1$ & $A_1^2$ \\
        \hline
        ${}^3 D_{4}$ & $G_2$ & $G_2$ & $A_1$ & $A_1^3$ \\
        \hline
        ${}^2 E_{6}$ & $F_4$ & $F_4$ & $A_1$ & $A_1^2$ \\
        \hline
    \end{tabular}
\end{center}


\subsubsection{}

Finally, let us discuss the action of $\rho$ on the weight lattice $\Lambda_{sc}$. Assume that $\Phi$ has one root length. Since $\rho$ permutes simple roots (hence all roots), the action of $\rho$ on root lattice $\Lambda_r$ is clear. The fundamental dominant weights $\lambda_1, \dots, \lambda_l$ forms a $\mathbb{Z}-$basis of the weight lattice $\Lambda_{sc}$. We can define the action of $\rho$ on $\lambda_i$ by $\rho (\lambda_i) = \lambda_j$ if $\rho (\alpha_i) = \alpha_j$. This action can be naturally extended to a $\mathbb{Z}-$linear automorphism $\rho$ of $\Lambda_{sc}$ such that $\rho (\Lambda_{r}) = \Lambda_r$. Thus $\rho$ can be thought as a group automorphism of the fundamental group $\Lambda_{sc}/\Lambda_r$ of $\Phi$. Now let $\Lambda$ be a sublattice of $\Lambda_{sc}$ which contains $\Lambda_r$. Then $\Lambda / \Lambda_r$ is a subgroup of $\Lambda_{sc} / \Lambda_r$ which is cyclic except for the case of $\Phi = D_{2n}$. Therefore $\rho(\Lambda / \Lambda_r) = \Lambda / \Lambda_r$ and hence $\rho (\Lambda) = \Lambda$. For the case of $\Phi = D_{2n}$, the fundamental group $\Lambda_{sc} / \Lambda_r$ is isomorphic to $\mathbb{Z}_2 \times \mathbb{Z}_2$. Hence there are exactly two proper sublattices $\Lambda_1$ and $\Lambda_2$ of $\Lambda_{sc}$ which contains $\Lambda_r$ as proper sublattice with the property that $\rho (\Lambda_i) \not\subset \Lambda_i$ for $i=1,2$. Therefore, if $\Lambda_\pi = \Lambda_1$ or $\Lambda_2$, then the graph automorphism of $G_\pi (\Phi, R)$ and $E_\pi (\Phi, R)$ do not exist even when $1/2 \in R$ (see \cite[page 91]{RS}).


\subsection{Twisted Chevalley Algebra}\label{subsec:TCA}

Let $\mathcal{L} = \mathcal{L}(\Phi, \mathbb{C})$ be a simple Lie algebra over $\mathbb{C}$ with root system $\Phi$. 
Assume that $\Phi$ is of type $A_n \ (n \geq 2)$, $D_n \ (n \geq 4)$, or $E_6$. 
Let $\Delta = \{ \alpha_1, \dots, \alpha_n \}$ be a fixed simple system.
Consider the Chevalley basis $\{X_\alpha, H_i \mid \alpha \in \Phi, \ i=1, \dots, n\}$. 
Let $\rho$ be a non-trivial angle-preserving permutation of $\Phi$, and write $\bar{\alpha} = \rho(\alpha)$ for all $\alpha \in \Phi$. 
This induces an automorphism of $\mathcal{L}$, also denoted by $\rho$, satisfying 
\[
    \rho (H_\alpha) = H_{\bar{\alpha}}, \quad 
    \rho (X_\alpha) = X_{\bar{\alpha}}, \quad 
    \rho (X_{-\alpha}) = X_{-\bar{\alpha}}, 
\]
for all $\alpha \in \Delta$. 
Then we have $\rho (X_\alpha) = \epsilon_\alpha X_{\bar{\alpha}}$, where $\epsilon_\alpha = \pm 1$ for all $\alpha \in \Phi$. 

\begin{lemma}[{\cite[Proposition 3.1]{EA1}}]\label{epsilonalpha}
    \normalfont
    A Chevalley basis of $\mathcal{L}$ can be chosen such that:
    \begin{enumerate}[(a)]
        \item $\epsilon_\alpha = \epsilon_{\bar{\alpha}}$, 
        \item $\epsilon_\alpha = -1$ if $[\alpha] \sim A_2$ and $\alpha = \bar{\alpha}$, 
        \item $\epsilon_\alpha = 1$ otherwise.
    \end{enumerate}
\end{lemma}

Let $R$ be a commutative ring with unity, and let $\theta: R \to R$ be an automorphism of $R$ of order $2$. 
Write $\bar{r} = \theta(r)$ for all $r \in R$. 
Define $R_\theta = \{r \in R \mid r = \bar{r}\}$ and $R_\theta^{-} = \{r \in R \mid r = -\bar{r}\}$. 
If $1/2 \in R$, then $R = R_\theta \oplus R_\theta^{-}$.

Let $\mathcal{L}(\Phi, R)$ denote a Chevalley algebra over $R$ of type $\Phi$ (cf. \ref{subsubsec:Chevalley algebra}). 
The automorphism $\theta$ induces a semi-automorphism of $\mathcal{L}(\Phi, R)$, also denoted by $\theta$, satisfying 
\[
    \theta (r H_i) = \bar{r} H_i \quad (i=1, \dots, n), \qquad 
    \theta (r X_\alpha) = \bar{r} X_\alpha \quad (\alpha \in \Phi), 
\]
for all $r \in R$.

Define $\sigma = \rho \circ \theta$. 
Then $\sigma$ is a semi-automorphism of $\mathcal{L}(\Phi, R)$ satisfying:
\[
    \sigma (r H_{\alpha_i}) = \bar{r} H_{\bar{\alpha_i}} \quad (i = 1, \dots, n), \qquad 
    \sigma (r X_\alpha) = \epsilon_\alpha \bar{r} X_{\bar{\alpha}} \quad (\alpha \in \Phi),
\]
for all $r \in R$. 
Define 
\[
    \mathcal{L}_{\sigma}(\Phi, R) = \{ X \in \mathcal{L}(\Phi, R) \mid \sigma(X) = X \}.
\]
Then $\mathcal{L}_\sigma(\Phi, R)$ forms a submodule of an $R_\theta$-module $\mathcal{L}(\Phi, R)$. 
This submodule is called the \textbf{twisted Chevalley algebra} over $R$ of type $\Phi$.


\subsubsection{\textbf{Basis of \texorpdfstring{$\mathcal{L}_\sigma (\Phi, R)$}{L(R)}}}

Suppose there exists an invertible element \( a \in R_\theta^{-} \). Furthermore, assume that \( 2 \) is invertible in \( R \). Let 
\[
\{ X_\alpha, H_i \mid \alpha \in \Phi, \ i=1, \dots, n \}
\]
denote the Chevalley basis as described in Lemma~\ref{epsilonalpha}.

Let \( \rho \) be a non-trivial angle-preserving permutation of \( \Phi \) of order \( 2 \), and let \( \Phi_\rho \) be the corresponding twisted Chevalley root system. Define the following elements:

\begin{align*}
    X^{+}_{[\alpha]} &= 
    \begin{cases}
        X_{\alpha}, & \text{if } [\alpha] \sim A_1, \\
        X_{\alpha} + X_{\bar{\alpha}}, & \text{if } [\alpha] \sim A_1^2 \text{ or } A_2;
    \end{cases} \\
    X^{-}_{[\alpha]}(\mathrm{I}) &= a (X_{\alpha} - X_{\bar{\alpha}}) \quad \text{if } [\alpha] \sim A_1^2 \text{ or } A_2; \\
    X^{-}_{[\alpha]}(\mathrm{II}) &= a X_{\alpha + \bar{\alpha}} \quad \text{if } [\alpha] \sim A_2; \\
    H^{+}_{[\alpha]} &= 
    \begin{cases}
        H_{\alpha}, & \text{if } [\alpha] \sim A_1, \\
        H_{\alpha} + H_{\bar{\alpha}}, & \text{if } [\alpha] \sim A_1^2 \text{ or } A_2;
    \end{cases} \\
    H^{-}_{[\alpha]} &= a (H_{\alpha} - H_{\bar{\alpha}}) \quad \text{if } [\alpha] \sim A_1^2 \text{ or } A_2.
\end{align*}

Then, the set
\[
    \{ X^{+}_{[\alpha]}, X^{-}_{[\alpha]}(\mathrm{I}), X^{-}_{[\alpha]}(\mathrm{II}), H^{+}_{[\alpha_i]}, H^{-}_{[\alpha_i]} \mid [\alpha] \in \Phi_\rho, \ [\alpha_i] \in \Delta_\rho \}
\]
forms a basis for the \( R_\theta \)-module \( \mathcal{L}_\sigma(\Phi, R) \). 
Finally, note that 
\[
    \mathcal{L}(\Phi, R) = R \otimes_{R_\theta} \mathcal{L}_\sigma(\Phi, R).
\]
Therefore, the same set also forms a basis for the \( R \)-algebra \( \mathcal{L}(\Phi, R) \).


\subsection{Twisted Chevalley Groups}\label{Subsec:TCG}

Assume that $\Phi$ is of type $A_n (n \geq 2), D_n (n \geq 4)$ or $E_6$ and let $G(R) = G_\pi (\Phi, R)$ (resp., $E(R) = E_\pi (\Phi, R)$) be a Chevalley group (resp., an elementary Chevalley group) over a commutative ring $R$. Let $\sigma$ be an automorphism of $G(R)$ which is the product of a graph automorphism $\rho$ and a ring automorphism $\theta$ such that $o(\theta) = o(\rho)$. Denote the corresponding permutation of the roots also by $\rho$. Since $\rho \circ \theta = \theta \circ \rho$, we have $o(\theta) = o(\rho) = o(\sigma)$. Since $E(R)$ is a characteristic subgroup of $G(R)$, $\sigma$ is also an automorphism of $E(R)$. 

Define $G_\sigma (R) = \{ g \in G(R) \mid \sigma(g)=g \}$. Clearly, $G_\sigma (R)$ is a subgroup of $G(R)$. We call $G_\sigma (R)$ the \textbf{twisted Chevalley group} over the ring $R$. The notion of the adjoint twisted Chevalley group and the universal (or simply connected) twisted Chevalley group is clear.

Write $E_\sigma (R) = E(R) \cap G_\sigma (R)$. Consider the subgroups $U, H, B, U^-$ and $N$ of $E(R)$. Then $\sigma$ preserves $U, H, B, U^-$ and $N$. Hence we can make sense of $U_\sigma, H_\sigma, B_\sigma, U^-_\sigma$ and $N_\sigma$ (if $A \subset G(R)$ then we define $A_\sigma = A \cap G_\sigma(R)$). Note that $\sigma$ preserves $N/H \cong W$ (as it preserves $N$ and $H$). The action thus induced on $W$ is concordant with the permutation $\rho$ of the roots. Finally, let us define $E_\sigma' (R) = \langle U_\sigma, U_\sigma^- \rangle$, a subgroup of $E_\sigma (R)$ generated by $U_\sigma$ and $U_\sigma^-$. We call $E'_\sigma (R)$ the \textbf{elementary twisted Chevalley group} over the ring $R$. Write $H'_\sigma = H \cap E'_\sigma(R), N'_\sigma = N \cap E'_\sigma(R)$ and $B'_\sigma = B \cap E'_\sigma (R)$. Then $B'_\sigma = U_\sigma H'_\sigma$. 

Let $T(R) = T_\pi (\Phi, R)$ be the standard maximal torus of $G(R)$. Define $T_\sigma (R) = T(R) \cap G_\sigma (R)$ and called it the \textbf{standard maximal torus} of $G_\sigma (R).$ 
For a character $\chi \in $ Hom$(\Lambda_\pi, R^*)$, we define its conjugation $\bar{\chi}_\sigma$ with respect to $\sigma$ by $\bar{\chi}_\sigma (\lambda) = \theta (\chi (\rho^{-1}(\lambda)))$ for every $\lambda \in \Lambda_\pi$. 
If $h(\chi) \in T(R)$, then $\sigma (h(\chi)) = h(\bar{\chi}_\sigma)$.
A character $\chi \in $ Hom$(\Lambda_\pi, R^*)$ is called \textbf{self-conjugate with respect to $\sigma$} if $\chi = \bar{\chi}_\sigma$, i.e., $\chi (\rho(\lambda)) = \theta (\chi (\lambda)),$ for every $\lambda \in \Lambda_\pi$. We denote the set of such characters by $\text{Hom}_1 (\Lambda_\pi, R^*) = \{ \chi \in \text{Hom} (\Lambda_\pi, R^*) \mid \chi = \bar{\chi}_\sigma \}$.
Thus we have $T_\sigma (R) = \{ h(\chi) \mid \chi \in \text{Hom}_1 (\Lambda_\pi, R^*) \}$. Note that, an element $h(\chi) \in H_\sigma \subset T_\sigma(R)$ if and only if $\chi$ is a self-conjugate character of $\Lambda_\pi$ (with respect to $\sigma$) that can be extended to a self-conjugate character of $\Lambda_{sc}$.

For the sake of completeness, let us also define $G_\sigma^0 (R) = G^0_\pi(\Phi, R) \cap G_\sigma (R)$ and $G'_\sigma (R) = T_\sigma (R) E'_\sigma (R)$. 

If $G(R)$ is of type $X$ and $\sigma$ is of order $n$, we say $G_\sigma(R)$ is of type $^nX.$ We write $G(R) \sim X$ and $G_\sigma(R) \sim {}^nX$. We use a similar notation for $E(R), E_\sigma(R)$ and $E'_\sigma(R)$. 

\begin{rmk}
    Sometimes we use more detailed notations such as $G_{\pi, \sigma} (\Phi, R)$ or $G_{\sigma} (\Phi, R)$ to refer to the group $G_\sigma(R)$. This convention similarly applies to other groups described above.
\end{rmk}


\subsection{List of Known Results}

We now present a collection of established results that will be utilized in the proof of our main theorem.


\subsubsection{\textbf{Automorphisms of Chevalley Algebras}}

We begin by presenting a result from \cite{AK} that investigate the $R$-algebra automorphisms of the Chevalley algebra $\mathcal{L}(\Phi, R)$. 

\vspace{2mm}

\noindent \textbf{Inner Automorphisms:} Let $G_{\text{ad}}(\Phi, R)$ be the elementary Chevalley group of type $\Phi$. By definition, each element of $G_{\text{ad}}(\Phi, R)$ is regarded as an automorphism of the Chevalley algebra $\mathcal{L}(\Phi, R)$. Such automorphisms are called \emph{inner automorphisms}.

\vspace{2mm}

\noindent \textbf{Graph automorphism:}
Let $\Phi$ be an irreducible root system and let $\Delta$ be a fixed simple system of $\Phi$. 
Consider a non-trivial angle-preserving permutation $\rho$ of the simple roots; such a $\rho$ induces an automorphism of the Coxeter graph (cf. Section~\ref{Subsec:TRS}). 
This automorphism naturally induces an automorphism of the Lie algebra $\mathcal{L}(\Phi, R)$, which is also denoted by $\rho$ (see \cite{AO&EV}). 

Let $\rho_1, \dots, \rho_k$ be all distinct automorphisms of $\mathcal{L}(\Phi, R)$ induced by the different symmetries of the given simple system $\Delta$ of $\Phi$ and let $R = R_1 \oplus \cdots \oplus R_k$ be a decomposition of the ring $R$ into a direct sum of ideals. Define an automorphism $\rho$ of the algebra 
$$\mathcal{L}(\Phi, R) = \mathcal{L}(\Phi, R_1) \oplus \cdots \oplus \mathcal{L}(\Phi, R_k)$$ 
by the rule 
\[
\rho(x_1 + \cdots + x_k) = \rho_1(x_1) + \cdots + \rho_k(x_k),
\]
where $x_i \in \mathcal{L}(\Phi, R_i)$. Such an automorphism $\rho$ is called a \emph{graph automorphism} of the algebra $\mathcal{L}(\Phi, R)$.
Clearly, the set of all graph automorphisms forms a group that is isomorphic to the group
\[
\mathcal{D}(\Phi, R) = \left\{ \sum e_i \rho_i \,\middle|\, e_i \in R, \, e_i^2 = e_i, \, e_i e_j = 0 \text{ for } i \neq j, \, \sum e_i = 1 \right\}.
\]

\begin{thm}[{\cite[Theorem 1]{AK}}]\label{thm:AK}
    \normalfont
    Let $R$ be a commutative ring with unity and $\Phi$ be an irreducible root system. Then every $R$-algebra automorphism of the Chevalley algebra $\mathcal{L}(\Phi, R)$ can be uniquely expressed as a composition of an inner automorphism and a graph automorphism. In particular,
    \[
        \text{Aut}_R(\mathcal{L}(\Phi, R)) \cong G_{\text{ad}}(\Phi, R) \rtimes \mathcal{D}(\Phi, R),
    \]
    where $\text{Aut}_R(\mathcal{L}(\Phi, R))$ denote the group of all $R$-algebra automorphisms of the Chevalley algebra $\mathcal{L}(\Phi, R)$. 
\end{thm}

\begin{rmk}
    In Theorem $1$ of \cite{AK}, A. Klyachko states the above result with certain restrictions on $R$. However, the same proof remains valid without these assumptions (see, for instance, \cite{EB12:main}).
\end{rmk}


\subsubsection{\textbf{Generation of $M_n(R)$ by Elementary Chevalley Groups}}

We now present a lemma from \cite{EB09:alwithhalfele}, which is stated as follows.

\begin{lemma}[{\cite[Lemma 8]{EB09:alwithhalfele}}]\label{lemma:EB09}
    The elementary Chevalley group $E_{\pi} (\Phi, R)$ generates the matrix ring $M_n(R)$.
\end{lemma}


\subsubsection{\textbf{\texorpdfstring{$E'_\sigma (R)$}{E(R)} is a characteristic subgroup of \texorpdfstring{$G_\sigma(R)$}{G(R)}}}

Recall that a subgroup $H$ of $G$ is called \textit{characteristic} if it is preserved under every automorphism of $G$, meaning that for any automorphism $f$ of $G$, we have $f(H) \subset H$. In the following 

\begin{thm}[{\cite[Theorem 11.1]{SG&DM1}}]\label{char subgrp}
    Let $\Phi_\rho$ be one of the following types: ${}^2 A_n \ (n \geq 3), {}^2 D_n \ (n \geq 4), {}^2 E_6$ or ${}^3 D_4$. Assume that $1/2 \in R$, and in addition, $1/3 \in R$ if $\Phi_\rho \sim {}^3 D_4$. 
    Additionally, assume that $R$ is a Noetherian ring.
    Let $H$ be a subgroup of $G_{\pi, \sigma} (\Phi, R)$ containing $E'_{\pi, \sigma} (\Phi, R)$. Then
    $E'_{\pi, \sigma} (\Phi, R)$ is a characteristic subgroup of $H$. In particular, $E'_{\pi, \sigma} (\Phi, R)$ is a characteristic subgroup of $G_{\pi, \sigma} (\Phi, R)$.
\end{thm}


\section{Tangent Algebra of Elementary Twisted Chevalley Groups}\label{sec:TA of ETCG}


In this section, we show that the twisted Chevalley algebra associated with elementary Chevalley groups can be reconstructed directly from the group itself. To achieve this, we introduce the concept of tangent algebra, following the approach developed by A. A. Klyachko in \cite{AK} for Chevalley groups.

Let $R$ be a commutative ring with unity. 
Assume that there exists an automorphism $\theta$ of $R$ of order two. 
This automorphism can be extended to the polynomial ring $R[t]$ with the variable $t$ as follows: $t \mapsto t$ and $r \mapsto \theta(r)$ for every $r \in R$. Consequently, the group $E'_{\text{ad}, \sigma}(\Phi, R[t])$ is well-defined and contains $E'_{\text{ad}, \sigma}(\Phi, R)$ as a subgroup.

Let $\mathcal{L}$ be a standard simple Lie algebra of type $\Phi$ over $\mathbb{C}$. Consider the twisted Chevalley algebra $\mathcal{L}_\sigma(\Phi, R)$ and the corresponding adjoint elementary twisted Chevalley group $E'_\sigma(R) = E'_{\text{ad}, \sigma}(\Phi, R)$. 
For $f(t) \in R_\theta[t]$, we define a map 
\[
    \text{REP}_f: E'_\sigma(R[t]) \longrightarrow E'_\sigma(R[t])
\]
that sends $g(t) \mapsto f(g(t))$. This map is then an endomorphism of the group $E'_\sigma(R)$.
Define a tangent space
$$T(E'_\sigma (R)) = \{ X \in M_n(R) \mid 1 + t X + t^2 Y \in E'_\sigma (R[t]) \text{ for some } Y \in M_n(R[t]) \}.$$
This set is an $R_\theta [E'_\sigma (R)]$-module, i.e., it is closed with respect to the following operations: 
\begin{enumerate}[(i)]
    \item Addition: $(1 + t X + o(t))(1+ t Y + o(t)) = 1 + t(X+Y) + o(t)$; 
    \item The scalar multiplication by $r \in R_\theta$: $\text{REP}_{rt}(1 + tX + o(t)) = 1 + r t X + o(t)$;
    \item The action of the group $E'_\sigma (R)$: $g(1 + t X + o(t))g^{-1} = (1+t gXg^{-1} + o(t))$ (henceforth, we define $g \circ X := g X g^{-1}$).
\end{enumerate}

Observe that the map $\sigma: E_{\text{ad}}(\Phi, R) \longrightarrow E_{\text{ad}}(\Phi, R)$ can be extended to a $R_\theta$-linear map from $M_n(R)$ onto itself, also denoted by the same symbol $\sigma$. 

\begin{prop}\label{prop:tangentspace of E'(R)}
    \normalfont
    $\mathcal{L}_\sigma (\Phi, R) = T(E'_\sigma (R)).$
\end{prop}

\begin{proof} 
    \noindent $(\subseteq):$ Consider the $R_\theta$-basis of $\mathcal{L}_\sigma(\Phi, R)$ as described in Section~\ref{subsec:TCA}. To establish the result, it suffices to show that each basis element lies in $T(E'_\sigma(R))$. We proceed as follows:
    
    \vspace{2mm}

    \noindent \textbf{Case $[\alpha] \sim A_1$ or $A_1^2$.}
    \begin{enumerate}[(a)]
        \item $X_{[\alpha]}^{+} \in T(E'_\sigma (R))$ because $x_{[\alpha]}(t) \in E'_\sigma (\Phi, R_\theta[t])$.
        \item $X_{[\alpha]}^{-} (\mathrm{I}) \in T(E'_\sigma (R))$ because $x_{[\alpha]}(a t) \in E'_\sigma (\Phi, R_\theta[t])$ (this is only in the case of $[\alpha] \sim A_1^2$).
        \item $H_{[\alpha]}^{+} \in T(E'_\sigma (R))$ because $H_{[\alpha]}^{+} = x_{[\alpha]} (1) \circ X_{-[\alpha]}^{+} + X_{[\alpha]}^{+} - X_{-[\alpha]}^{+}$.
        \item $H_{[\alpha]}^{-} \in T(E'_\sigma (R))$ because $H_{[\alpha]}^{-} = x_{[\alpha]} (1) \circ X_{-[\alpha]}^{-} (\mathrm{I}) + X_{[\alpha]}^{-}(\mathrm{I}) - X_{-[\alpha]}^{-}(\mathrm{I})$ (this is only in the case of $[\alpha] \sim A_1^2$).
    \end{enumerate}

    \noindent \textbf{Case $[\alpha] \sim A_2$.}
    \begin{enumerate}[(a)]
        \item $X_{[\alpha]}^{+} \in T(E'_\sigma (R))$ because $x_{[\alpha]}(t, t^2/2) \in E'_\sigma (\Phi, R_\theta[t])$.
        \item $X_{[\alpha]}^{-} (\mathrm{I}) \in T(E'_\sigma (R))$ because $x_{[\alpha]}(at, (at)^2/2) \in E'_\sigma (\Phi, R_\theta[t])$.
        \item $X_{[\alpha]}^{-} (\mathrm{II}) \in T(E'_\sigma (R))$ because $x_{[\alpha]}(0, at) \in E'_\sigma (\Phi, R_\theta[t])$.
        \item $H_{[\alpha]}^{+} \in T(E'_\sigma (R))$ because $H_{[\alpha]}^{+} = x_{[\alpha]} (1, 1/2) \circ X_{-[\alpha]}^{+} + \frac{1}{2} X_{[\alpha]}^{+} - X_{-[\alpha]}^{+}$.
        \item $H_{[\alpha]}^{-} \in T(E'_\sigma (R))$ because $H_{[\alpha]}^{-} = x_{[\alpha]} (1, 1/2) \circ X_{-[\alpha]}^{-} + \frac{3}{2} X_{[\alpha]}^{-} (\mathrm{I}) - X^{-}_{[\alpha]}(\mathrm{II}) - X^{-}_{-[\alpha]}(\mathrm{I})$.
    \end{enumerate}

    \noindent $(\supseteq):$ 
    Let $X \in T(E'_\sigma (R))$, i.e., $1 + t X + o(t) \in E'_\sigma (R[t])$. Write it as a product of elementary generators:
    \begin{equation}\label{eq_6.2.1}
        1 + t X + o(t) = \prod_{j=1}^m x_{[\alpha_j]}(f_j(t)),
    \end{equation}
    where $f_j (t) = r_j \cdot t^{k_j}$ if $[\alpha_j] \sim A_1$ or $A_1^2$; $(r_j^{(1)} t^{k_j^{(1)}}, r_j^{(2)} t^{k_j^{(2)}})$ if $[\alpha] \sim A_2$.
    Without loss of generality, we can assume that $k_j \in \{ 0, 1 \}$ and $(k_{j}^{(1)}, k_{j}^{(2)}) \in \{ (0,0), (1,2), (0,1) \}$.
    Let $\mathcal{J}'$ be the subset of $\mathcal{J} = \{1, \dots, m\}$ consisting of all $j$ such that $k_j = 0$ or $k_j^{(1)} = k_j^{(2)} = 0$.
    Now, for all $j \in \mathcal{J} \setminus \mathcal{J}'$, define
    \[
        g_j = \prod_{i \in \mathcal{J}_j} x_{[\alpha_i]}(f_i (t)) = \prod_{i \in \mathcal{J}_j} x_{[\alpha_i]}(f_i (0)) \in E'_\sigma (R),
    \]
    where $\mathcal{J}_i = \mathcal{J}' \cap \{ 1, \dots, j\}$.
    By putting $t=0$ in equation (\ref{eq_6.2.1}), we obtain
    \[
        g_m = \prod_{j \in \mathcal{J}'} x_{[\alpha_j]}(f_j (0)) = 1.
    \]
    Using this, we can rewrite equation (\ref{eq_6.2.1}) as
    \[
        1 + tX + o(t) = \prod_{j \in \mathcal{J} \setminus \mathcal{J}'} g_j x_{[\alpha_j]}(f_j(t)) g_j^{-1}.
    \]
    By comparing both sides, we conclude that $X = \displaystyle\sum_{j \in \mathcal{J} \setminus \mathcal{J}'} g_j \circ (r_j X^{\pm}_{[\alpha_j]}) \in \mathcal{L}_\sigma (\Phi, R)$, as desired.
\end{proof}


\section{Key Lemma}\label{sec:key lemma}

We now introduce a lemma that plays a crucial role in establishing the main result of this chapter.

\begin{lemma}\label{key lemma}
    The set $E'_{\pi, \sigma} (\Phi, R)$ generates $M_n (R)$ as an $R$-algebra, where $n$ is the dimension of the corresponding representation $\pi$ of Lie algebra $\mathcal{L}$. 
\end{lemma}

\begin{proof}
        Let $\mathcal{M}$ be the $R$-subalgebra of $M_n (R)$ generated by the set $E'_\sigma (R)$. To prove the lemma we must show that $\mathcal{M} = M_n (R)$.
        We now assume that $\pi$ is a minimal representation (the definition of minimal representation can be found in Section 3 of Chapter 3 in \cite{NV1}). Therefore $\pi (X_\alpha)^3 = 0$ for all $\alpha \in \Phi$ (note that we only consider the case where $\Phi \sim A_n, D_n$, or $E_6$). In particular, we have 
        \[
            x_{\alpha}(t) = 1 + t \, \pi(X_\alpha) + \frac{t^2 \, \pi (X_\alpha)^2}{2!}.
        \]

        We now claim that the elements $\pi (X_\alpha) \in \mathcal{M}$ for every $\alpha \in \Phi$.
        To see this, let $\alpha \in \Phi$ and let $[\alpha]$ denote the corresponding element in $\Phi_\rho$. 
        Just for our convenience, we use the notation $X_{\alpha}$ to denote the linear map $\pi(X_{\alpha})$.
    
        Suppose $[\alpha] \sim A_1$, then 
        \[
            x_{[\alpha]}(t) = x_{\alpha}(t) = 1 + t \, X_\alpha + \frac{t^2 \, (X_\alpha)^2}{2!}.
        \] 
        Then we can easily check that 
        \[
            X_\alpha = \frac{(x_{[\alpha]}(1) - x_{[\alpha]}(-1))}{2}.
        \] 
    
        Now suppose $[\alpha] \sim A_1^2$, then 
        \begin{equation*}
            \begin{split}
                x_{[\alpha]}(t) = x_{\alpha} (t) x_{\bar{\alpha}}(\bar{t}) = 1 + t \, X_\alpha + \bar{t} \, X_{\bar{\alpha}} + \frac{t^2 \, (X_\alpha)^2}{2} + \frac{(\bar{t})^2 \, (X_{\bar{\alpha}})^2}{2} + t \bar{t} \, (X_\alpha \, X_{\bar{\alpha}}) \\
                \frac{t (\bar{t})^2 \, (X_\alpha \, (X_{\bar{\alpha}})^2)}{2} + \frac{(t^2) \bar{t} \, ((X_\alpha)^2 \, X_{\bar{\alpha}})}{2} + \frac{(t \bar{t})^2 \, ((X_\alpha)^2 \, (X_{\bar{\alpha}})^2)}{4}.
            \end{split}
        \end{equation*}
        A simple calculation on roots shows that 
        \[
            X_\alpha \, (X_{\bar{\alpha}})^2 
            = (X_\alpha)^2 \, X_{\bar{\alpha}} 
            = (X_\alpha)^2 \, (X_{\bar{\alpha}})^2
            = 0.
        \]
        Therefore, 
        \begin{equation*}
            x_{[\alpha]}(t) = 1 + 
            \Big( t \, X_\alpha + \bar{t} \, X_{\bar{\alpha}} \Big) + 
            \frac{1}{2} \Big( t \, X_\alpha + \bar{t} \, X_{\bar{\alpha}} \Big)^2.
        \end{equation*}
        Then we have 
        \begin{align*}
            X_{\alpha} + X_{\bar{\alpha}} &= \frac{1}{2} (x_{[\alpha]}(1) - x_{[\alpha]}(-1)), \text{ and} \\
            X_{\alpha} - X_{\bar{\alpha}} &= \frac{1}{2a} (x_{[\alpha]}(a) - x_{[\alpha]}(-a)).
        \end{align*}
        Hence 
        \begin{align*}
            X_{\alpha} &= \frac{1}{4} \Big( \big(x_{[\alpha]}(1) - x_{[\alpha]}(-1) \big) + a^{-1} \big(x_{[\alpha]}(a) - x_{[\alpha]}(-a) \big) \Big), \text{ and} \\
            X_{\bar{\alpha}} &= \frac{1}{4} \Big( \big(x_{[\alpha]}(1) - x_{[\alpha]}(-1) \big) - a^{-1} \big(x_{[\alpha]}(a) - x_{[\alpha]}(-a) \big) \Big).
        \end{align*}

        Finally suppose $[\alpha] \sim A_2$, then 
        \begin{align*}
            x_{[\alpha]}(t,u) &= x_{\alpha} (t) x_{\bar{\alpha}}(\bar{t}) x_{\alpha + \bar{\alpha}} (N_{\bar{\alpha}, \alpha} u) \\
            &= 1 + \Big( t \, X_\alpha + \bar{t} \, X_{\bar{\alpha}} + N_{\bar{\alpha}, \alpha} u \, X_{\alpha + \bar{\alpha}} \Big) 
            + \Big( \frac{t^2}{2} \, (X_\alpha)^2 + \frac{(\bar{t})^2}{2} \, (X_{\bar{\alpha}})^2 + \frac{u^2}{2} \, (X_{\alpha + \bar{\alpha}})^2 \\
            & \hspace{5mm} + t \bar{t} \, (X_\alpha \, X_{\bar{\alpha}}) + N_{\bar{\alpha},\alpha} t u \, (X_\alpha \, X_{\alpha + \bar{\alpha}}) + N_{\bar{\alpha},\alpha} \bar{t} u \, (X_{\bar{\alpha}} \, X_{\alpha + \bar{\alpha}}) \Big) \\
            & \hspace{5mm} + \Big( N_{\bar{\alpha}, \alpha} t \bar{t} u \, (X_\alpha \, X_{\bar{\alpha}} \, X_{\alpha+ \bar{\alpha}}) + \frac{t^2 \bar{t}}{2} \, ((X_\alpha)^2 \, X_{\bar{\alpha}}) + \frac{t (\bar{t})^2}{2} \, (X_\alpha \, (X_{\bar{\alpha}})^2) \Big).
        \end{align*}
        Note that all the other terms in the expansion of the product $x_{\alpha} (t) x_{\bar{\alpha}}(\bar{t}) x_{\alpha + \bar{\alpha}} (N_{\bar{\alpha}, \alpha} u)$ which does not appear in the above expression are $0$ (for instance $X_\alpha \, (X_{\alpha + \bar{\alpha}})^2 = 0, (X_\alpha)^2 \, X_{\alpha + \bar{\alpha}} = 0,$ etc.). 
        First note that, 
        \[
            X_{\alpha + \bar{\alpha}} = \frac{N_{\bar{\alpha},\alpha}}{2a} \Big( X_{[\alpha]}(0,a) - x_{[\alpha]}(0, -a) \Big).
        \]
        Now define
        \begin{align*}
            \mathcal{F}_{[\alpha]}(t, u) &= x_{[\alpha]}(t,u) - x_{[\alpha]}(-t, u) \\
            &= 2(t X_\alpha + \bar{t} X_{\bar{\alpha}}) + 2 u (t X_\alpha X_{\alpha + \bar{\alpha}} + \bar{t} X_{\bar{\alpha}} X_{\alpha + \bar{\alpha}}) + t \bar{t} (t (X_\alpha)^2 X_{\bar{\alpha}} + \bar{t} X_{\alpha} (X_{\bar{\alpha}})^2).
        \end{align*}
        Then 
        \begin{align*}
            X_{\alpha} + X_{\bar{\alpha}} &= \frac{1}{12} \big( 8 \mathcal{F}(1, 1/2) - \mathcal{F}(2, 2) \big), \text{ and } \\
            X_{\alpha} - X_{\bar{\alpha}} &= \frac{1}{12 a} \big( 8 \mathcal{F}(a, -\frac{a^2}{2}) - \mathcal{F}(2a, -2a^2) \big).
        \end{align*}
        Hence
        \begin{align*}
            X_{\alpha} &= \frac{1}{24} \Big( \big( 8 \mathcal{F}(1, 1/2) - \mathcal{F}(2, 2) \big) + a^{-1} \big( 8 \mathcal{F}(a, -\frac{a^2}{2}) - \mathcal{F}(2a, -2a^2) \big) \Big), \text{ and } \\
            X_{\bar{\alpha}} &= \frac{1}{24} \Big( \big( 8 \mathcal{F}(1, 1/2) - \mathcal{F}(2, 2) \big) - a^{-1} \big( 8 \mathcal{F}(a, -\frac{a^2}{2}) - \mathcal{F}(2a, -2a^2) \big) \Big).
        \end{align*}
    
        Therefore, $\pi (X_{\alpha}) \in \mathcal{M}$ for every $\alpha \in \Phi$. 
        In particular, $\pi (\mathcal{L}(R)) \subset \mathcal{M}$. 
        Hence, $E_\pi(\Phi, R) \subset \mathcal{M}$. 
        By Lemma~\ref{lemma:EB09}, the set $E(R)$ generates $M_n(R)$ as an $R$-algebra. Therefore, $\mathcal{M} = M_n(R)$. 
        This completes the proof of the lemma.
\end{proof}


\section{Normalizer of \texorpdfstring{$G_{\pi,\sigma}(\Phi, R)$}{G(R)} and \texorpdfstring{$E'_{\pi,\sigma}(\Phi, R)$}{E(R)} in \texorpdfstring{$G_{\pi,\sigma}(\Phi, S)$}{G(S)}}\label{sec:normalizer_of_G_and_E}

This section aims to present a proof of the main theorem. We begin by proving that $N_{G_{\pi, \sigma} (\Phi, S)}(G_{\pi, \sigma} (\Phi, R)) = N_{G_{\pi, \sigma} (\Phi, S)}(E'_{\pi, \sigma} (\Phi, R))$. 
Since $E'_{\pi, \sigma} (\Phi, R)$ is a characteristic subgroup of $G_{\pi,\sigma} (\Phi, R)$ (see Theorem~\ref{char subgrp}), it follows that 
$$ N_{G_{\pi, \sigma} (\Phi, S)}(G_{\pi, \sigma} (\Phi, R)) \subset N_{G_{\pi,\sigma} (\Phi, S)}(E'_{\pi, \sigma} (\Phi, R)). $$ 
    
For the reverse inclusion, let $g \in N_{G_{\pi,\sigma} (\Phi, S)} (E'_{\pi, \sigma} (\Phi, R))$. 
Thus, we have $g E'_{\pi, \sigma} (\Phi, R) g^{-1} = E'_{\pi, \sigma} (\Phi, R).$
By Lemma~\ref{key lemma}, we have $g M_n (R) g^{-1} = M_n (R)$. 
This implies that $g GL_n (R) g^{-1} = GL_n (R)$. 
In particular, $g G_{\pi, \sigma} (\Phi, R) g^{-1} \subset GL_n (R)$.
On the other hand, $g G_{\pi, \sigma} (\Phi, R) g^{-1} \subset G_{\pi, \sigma} (\Phi, S)$ as $g \in G_{\pi, \sigma} (\Phi, S)$.
Therefore, we have 
\[
    g G_{\pi, \sigma}(\Phi, R) g^{-1} \subset GL_n(R) \cap G_{\pi, \sigma} (\Phi, S) = G_{\pi, \sigma} (\Phi, R).
\]
Thus, $g \in N_{G_{\pi, \sigma} (\Phi, S)}(G_{\pi, \sigma} (\Phi, R))$, which implies that $$N_{G_{\pi, \sigma} (\Phi, S)}(E'_{\pi, \sigma} (\Phi, R)) \subset N_{G_{\pi, \sigma} (\Phi, S)}(G_{\pi, \sigma} (\Phi, R)),$$ as desired.  

Now, it remains to prove that $N_{G_{\text{ad}, \sigma} (\Phi, S)} (E'_{\text{ad}, \sigma} (\Phi, R)) = G_{\text{ad}, \sigma} (\Phi, R).$ 
Clearly, $$G_{\text{ad}, \sigma} (\Phi, R) \subset N_{G_{\text{ad}, \sigma} (\Phi, S)} (G_{\text{ad}, \sigma} (\Phi, R)) = N_{G_{\text{ad}, \sigma} (\Phi, S)} (E'_{\text{ad}, \sigma} (\Phi, R)).$$ 
To establish the reverse inclusion, let $g \in N_{G_{\text{ad}, \sigma} (\Phi, S)} (E'_{\text{ad}, \sigma} (\Phi, R))$. 
Then, by definition, $$g E'_{\text{ad}, \sigma} (\Phi, R) g^{-1} = E'_{\text{ad}, \sigma} (\Phi, R).$$ 
Using Proposition~\ref{prop:tangentspace of E'(R)}, we conclude that $$g (\text{ad}(\mathcal{L}_\sigma (\Phi, R))) g^{-1} = \text{ad}(\mathcal{L}_\sigma (\Phi, R)).$$
By ``complexification", we get $$g (\text{ad}(\mathcal{L} (\Phi, R))) g^{-1} = \text{ad}(\mathcal{L} (\Phi, R)).$$
That is, $i_g$ is an automorphism of the Lie algebra $\text{ad}(\mathcal{L}(\Phi, R))$. 
Thus, $i_g$ can be decomposed as $i_g = i_{g'} \circ \delta$, where $g' \in G_{\text{ad}}(\Phi, R)$ and $\delta$ is a graph automorphism (see Theorem~\ref{thm:AK}). 
Consequently, $i_g$ is an automorphism of $E_{\text{ad}} (\Phi, R)$ that admits the same decomposition.
Since the intersection between the set of all graph automorphisms and the set of all inner automorphisms of $E_{\text{ad}} (\Phi, R)$ is trivial (see \cite{EA4}), we deduce that $\delta = \operatorname{id}$.
It follows that $g \cdot {g'}^{-1} = \lambda \cdot 1_n \in G_{\text{ad}}(\Phi, S)$ for some $\lambda \in S$. 
Since there are no scalar matrices in the group of adjoint type (as the centre is trivial, see \cite{EA&JH}), we conclude that $\lambda = 1$. 
Consequently, $g = g' \in G_{\text{ad}, \sigma}(\Phi, S) \cap G_{\text{ad}}(\Phi, R) = G_{\text{ad}, \sigma}(\Phi, R)$.
Therefore, $N_{G_{\text{ad}, \sigma} (\Phi, S)} (E'_{\text{ad}, \sigma} (\Phi, R)) \subset G_{\text{ad}, \sigma} (\Phi, R)$, as desired.
This completes the proof of Theorem~\ref{thm:normalizer_of_G_and_E}.

\begin{rmk}
    The Noetherian condition is only required to apply Theorem \ref{char subgrp}.
\end{rmk}



\end{document}